\documentclass[11pt]{article}

\usepackage{amsfonts}
\usepackage{amsmath}
\usepackage{amssymb}
\usepackage{amscd}
\usepackage{amsthm}
\usepackage{pstricks}

\newtheorem{theorem}{Theorem}[section]
\newtheorem{lemma}[theorem]{Lemma}
\theoremstyle{definition}
\newtheorem {definition}[theorem]{Definition}

\theoremstyle{remark}
\newtheorem{remark}[theorem]{Remark}

\def\ra{\rightarrow}
\def\iy{\infty}

\def\be{\begin{equation}}
\def\ee{\end{equation}}
\def\ba{\begin{eqnarray*}}
\def\ea{\end{eqnarray*}}
\def\bae{\begin{eqnarray}}
\def\eae{\end{eqnarray}}
\def\bc{\begin{center}}
\def\ec{\end{center}}

\begin{document}
\title{Limit Distributions of Eigenvalues for Random Block Toeplitz and Hankel Matrices}

\date{March 18, 2010}
\author{Yi-Ting Li,\ Dang-Zheng Liu \ and
Zheng-Dong Wang\\
School of Mathematical Sciences\\
Peking University\\
Beijing, 100871, P. R. China }

\maketitle

\begin{abstract}
Block Toeplitz and Hankel matrices arise in many aspects of
applications. In this paper, we will research the distributions of
eigenvalues for some models and get the semicircle law. Firstly we
will give trace formulae of block Toeplitz and Hankel matrix. Then
we will prove that the almost sure limit $\gamma_{_T}^{(m)}$
$(\gamma_{_H}^{(m)})$ of eigenvalue distributions of random block
Toeplitz (Hankel) matrices exist and give the moments of the limit
distributions where $m$ is the order of the blocks. Then we will
prove the existence of almost sure limit of eigenvalue distributions
of random block Toeplitz and Hankel band matrices and give the
moments of the limit distributions. Finally we will prove that
$\gamma_{_T}^{(m)}$ $(\gamma_{_H}^{(m)})$ converges weakly to the
semicircle law as $m\ra\iy$.
\end{abstract}



\section{Introduction}
In random matrix theory, a very important object is the eigenvalue
distribution of a random matrix. If $A=(a_{ij}(\omega))_{i,j=1}^N$
is a real symmetric random matrix where the $(a_{ij}(\omega))$'s are
random variables on a probability space $\Omega$ with a probability
measure $P$, then the eigenvalue distribution of $A$ is
\[
\mu_{_A}=\frac{1}{N}\int_{\Omega}\sum\limits_{j=1}^N\delta_{\lambda_j(\omega)}dP(\omega)
\]
where $\lambda_j(\omega)$'s are the $N$ real eigenvalues of $A$.

The asymptotic behavior of the eigenvalue distribution is of much
importance. In [23,24], Wigner got the semicircle law for a wide
class of real symmetric random matrices and this great result caused
much development of random matrix theory. Recently in a review paper
(see [1]), Bai proposed the study of random matrix models with
certain additional linear structure. The properties of the
distributions of eigenvalues for random Hankel and Toeplitz matrices
with independent entries are listed among the unsolved random matrix
problems posed in [1]. In [5], Bryc, Dembo and Jiang proved the
existence of limit distribution $\gamma_{_T}$ and $\gamma_{_H}$ of
real symmetric Toeplitz and Hankel matrices. The moments of
$\gamma_{_T}$ and $\gamma_{_H}$ are the sum of volumes of solids.
Hammond and Miller in [9] also proved the existence of $\gamma_{_T}$
and $\gamma_{_H}$ independently. In [15], Liu and Wang proved the
existence of limit distribution $\gamma_{_T}$ and $\gamma_{_H}$ for
real symmetric, complex Hermitian band Toeplitz and real symmetric
band Hankel matrices. We notice that Basak and Bose (see [2]) and
Kargin (see [13]) also did the same work independently. Especially,
the limit distribution of random Toeplitz band matrices with
bandwidth $b_{N}=o(N)$ is Gaussian.

Block Toeplitz and Hankel matrices arise in many aspects of
mathematics, physics and technology (see [8,10,14,19]). A block
Toeplitz matrix is a block matrix which can be written as
\[
T=(A_{i-j})_{i,j=1}^N=
\begin{array}{ccc}
\left( \begin{array}{cccccc}
A_{0}&A_{-1}&A_{-2}&\cdots&A_{-(N-1)}\\
A_{1}&A_{0}&A_{-1}&\cdots&A_{-(N-2)}\\
A_{2}&A_{1}&A_{0}&\cdots&A_{-(N-3)}\\
\vdots&\vdots&\vdots&\ddots &\vdots\\
A_{N-1}&A_{N-2}&A_{N-3}&\cdots&A_{0}
\end{array}\right)
\end{array}
\]
where $A_s=(a_{ij}(s))_{i,j=1}^m$ is an $m\times m$ matrix, $\forall
s\in\{-N+1,...,N-1\}$. In [7], Gazzah, Regalia and Delmas researched
the asymptotic behavior of the eigenvalue distribution for block
Toeplitz matrices. In [18], Rashidi Far, Oraby, Bryc and Speicher
have proved the existence of the limit distribution of eigenvalues
for a random block Toeplitz matrix whose blocks are selfadjoint
$m\times m$ matrices as $m\ra\iy$, which implies that the double
limit $\lim\limits_{N\ra\iy}\lim\limits_{m\ra\iy}$ gives the
semicircle law. In this paper, we will study the limit distributions
of real symmetric random block Toeplitz and Hankel matrices as
$\lim\limits_{N\ra\iy}$ and
$\lim\limits_{m\ra\iy}\lim\limits_{N\ra\iy}$.

In a block Toeplitz matrix with the form mentioned above, we suppose
the $a_{ij}(s)$'s are real random variables.
For symmetry, we need $A_s=(A_{-s})^T$. For independence of the elements, we suppose:\\
(1) $a_{i_1j_1}(s_1)$ and $a_{i_2j_2}(s_2)$ are independent if
$|s_1|\ne|s_2|$, \\
(2) If $s\ne0$ and $(i_1,j_1)\ne(i_2,j_2)$ then $a_{i_1j_1}(s)$ and
$a_{i_2j_2}(s)$ are independent, \\
(3) If $(i_1,j_1)\ne(i_2,j_2)$ and $(i_1,j_1)\ne(j_2,i_2)$ then
$a_{i_1j_1}(0)$ and $a_{i_2j_2}(0)$ are independent.\\
In addition, we need the following uniform boundedness conditoin:\\
(4)
\begin{eqnarray}
E(a_{ij}(s))=0,\,E(|a_{ij}(s)|^2)=1,-(N-1)\le s\le N-1,1\le i,j\le m
\end{eqnarray}
and
\begin{eqnarray}
\sup\limits_{\mbox{\tiny
$\begin{array}{c}N\in \mathbb{N}\\
-(N-1)\le s\le N-1\end{array}$}}\big\{|a_{ij}(s)|^k\big|1\le i,j\le
m\big\}=C_{k,m}<+\iy
\end{eqnarray}
and
\begin{eqnarray}
\sup\limits_{m\in \mathbb{N}}C_{k,m}=C_k<+\iy.
\end{eqnarray}
A block Hankel matrix is a block matrix which can be written as
\[
H=(A_{N+1-i-j})_{i,j=1}^N=
\begin{array}{ccc}
\left( \begin{array}{cccccc}
A_{N-1}&A_{N-2}&A_{N-3}&\cdots&A_{0}\\
A_{N-2}&A_{N-3}&A_{N-4}&\cdots&A_{-1}\\
A_{N-3}&A_{N-4}&A_{N-5}&\cdots&A_{-2}\\
\vdots&\vdots&\vdots&\ddots &\vdots\\
A_{0}&A_{-1}&A_{-2}&\cdots&A_{-(N-1)}
\end{array}\right)
\end{array}
\]
where $A_s=(a_{ij}(s))_{i,j=1}^m$ is an $m\times m$ matrix, $\forall
s\in\{-N+1,...,N-1\}$. Similar to random block Toeplitz matrices, we
suppose the $a_{ij}(s)$'s are real random variables, $A_s=(A_s)^T$.
In addition, we assume:\\
(1) $a_{i_1j_1}(s_1)$ and $a_{i_2j_2}(s_2)$ are independent if
$s_1\ne s_2$, \\
(2) If $(i_1,j_1)\ne(i_2,j_2)$ and $(i_1,j_1)\ne(j_2,i_2)$ then
$a_{i_1j_1}(s)$ and $a_{i_2j_2}(s)$ are independent,\\
(3) \[ E(a_{ij}(s))=0,\,E(|a_{ij}(s)|^2)=1,-(N-1)\le s\le N-1,1\le
i,j\le m
\]
and
\[
\sup\limits_{\mbox{\tiny
$\begin{array}{c}N\in \mathbb{N}\\
-(N-1)\le s\le N-1\end{array}$}}\big\{|a_{ij}(s)|^k\big|1\le i,j\le
m\big\}=C_{k,m}<+\iy
\]
and
\[
\sup\limits_{m\in \mathbb{N}}C_{k,m}=C_k<+\iy.
\]

We will firstly give trace formulae of block Toeplitz and Hankel
matrices in Section 2. Using those trace formulae, we will prove
that the almost sure limit distributions of block Toeplitz and
Hankel matrices exist in Section 3. In that section we will also
give the moments of the limit distributions. In Section 4, we will
prove the existence of almost sure limit distributions of block
Toeplitz and Hankel band matrices and give the moments of the limit
distributions. In Section 5 we will prove that for block Toeplitz
and Hankel matrices, the double limit
$\lim\limits_{m\ra\iy}\lim\limits_{N\ra\iy}$ gives the semicircle
law.

\section{Trace Formulae of Block Toeplitz and Hankel Matrices}

\begin{definition} Let $T$ be an $mN \times mN$ matrix and consist of
$N^2$ blocks. If $T$ has the form
\[
T=(A_{i,j})_{i,j=1}^{N}=
\begin{array}{ccc}
\left( \begin{array}{cccccc}
A_{0}&A_{-1}&A_{-2}&\cdots&A_{-(N-1)}\\
A_{1}&A_{0}&A_{-1}&\cdots&A_{-(N-2)}\\
A_{2}&A_{1}&A_{0}&\cdots&A_{-(N-3)}\\
\vdots&\vdots&\vdots&\ddots &\vdots\\
A_{N-1}&A_{N-2}&A_{N-3}&\cdots&A_{0}
\end{array}\right)
\end{array}
\]
where $\{A_{-(N-1)},...,A_0,...,A_{N-1}\}$ is a set of $m\times m$
matrices and $A_{i,j}=A_{i-j}=(a_{pq}(i-j))_{p,q=1}^m$ , then we
call $T$ a block Toeplitz matrix.
\end{definition}

Let $H$ be an $mN \times mN$ matrix and can be written as $H=\Phi T$
where $T$ is a block Toeplitz matrix and
\[
\Phi=
\begin{array}{ccc}
\left( \begin{array}{cccccc}
0&\cdots&0&I_m\\
0&\cdots&I_m&0\\
\vdots&\ddots &\vdots&\vdots\\
I_m&\cdots&0&0
\end{array}\right)
\end{array}
\]
where $I_m$ is the $m\times m$ unit matrix. Then we call $H$ a block
Hankel matrix.

For convenience, let $b_N=N-1$.
\begin{lemma}Let $T=(A_{i-j})_{i,j=1}^{N}$ be a block Toeplitz
matrix and $A_{s}=(a_{pq}(s))_{p,q=1}^m$ where $-b_N\le s\le
b_N;\,1\le p,q\le m$. Then we have a trace formula
\begin{eqnarray*}
\mathrm{tr}(T^k)&=&\sum_{i=1}^N
\sum_{j_1,...,j_k=-b_N}^{b_N}\mathrm{tr}(A_{j_1}\cdots
A_{j_k})\prod_{l=1}^k I_{[1,N]}(i+\sum_{q=1}^l
j_q)\delta_{0,\sum\limits_{q=1}^k j_q}\\
&=&\sum_{i=1}^N \sum_{j_1,...,j_k=-b_N}^{b_N}\sum_{t_1,...,t_k=1}^m
a_{t_1t_2}(j_1)\cdots a_{t_kt_1}(j_k)\prod_{l=1}^k
I_{[1,N]}(i+\sum_{q=1}^l j_q)\delta_{0,\sum\limits_{q=1}^k j_q}.
\end{eqnarray*}
\end{lemma}

\begin{lemma}Let $H=\Phi T$ be a block
Hankel matrix and
$T=(A_{i-j})_{i,j=1}^{N};\,A_{s}=(a_{pq}(s))_{p,q=1}^m$ where
$-b_N\le s\le b_N;\,1\le p,q\le m$. Then we have the trace formula
\[
\mathrm{tr}(H^k)=
\]
\[
\begin{cases} \sum\limits_{i=1}^N
\sum\limits_{j_1,...,j_k=-b_N}^{b_N}\mathrm{tr}(A_{j_1}\cdots
A_{j_k})\prod\limits_{l=1}^k I_{[1,N]}(i-\sum\limits_{q=1}^l
(-1)^q j_q)\delta_{0,\sum\limits_{q=1}^k (-1)^q j_q}&\text{$k$ even}\\
\sum\limits_{i=1}^N\sum\limits_{j_1,...,j_k=-b_N}^{b_N}\mathrm{tr}(A_{j_1}\cdots
A_{j_k})\prod\limits_{l=1}^k I_{[1,N]}(i-\sum\limits_{q=1}^l (-1)^q
j_q)\delta_{2i-1-N,\sum\limits_{q=1}^k (-1)^q j_q}&\text{$k$ odd}
\end{cases}
\]
where
\[
 \mathrm{tr}(A_{j_1}\cdots A_{j_k})=\sum\limits_{t_1,...,t_k=1}^m
a_{t_1t_2}(j_1)\cdots a_{t_kt_1}(j_k).
\]
\end{lemma}

To prove the above lemmas, we consider Kronecker product of two
matrices (see [12]). Let $A=(a_{ij})$ be an $m\times n$ matrix and
$B$ be a $p\times q$ matrix. The Kronecker product of $A$ and $B$ is
an $mp\times nq$ matrix:
\[
A\otimes B=
\begin{array}{ccc}
\left( \begin{array}{cccccc}
a_{11}B&\cdots&a_{1n}B\\
\vdots&\ddots &\vdots\\
a_{m1}B&\cdots&a_{mn}B
\end{array}\right)
\end{array}.
\]

\begin{proof}[Proof of Lemma 2.2]
Let B and F be two $N\times N$ matrices and
\[
B=(\delta_{i+1,j})_{i,j=1}^{N}=
\begin{array}{ccc}
\left( \begin{array}{cccccc}
0&1&0&\cdots&0&0\\
0&0&1&\cdots&0&0\\
0&0&0&\cdots&0&0\\
\vdots&\vdots&\vdots&\ddots &\vdots&\vdots\\
0&0&0&\cdots&0&1\\
0&0&0&\cdots&0&0
\end{array}\right)
\end{array}
,\\
 F=(\delta_{i,j+1})_{i,j=1}^{N}=
\begin{array}{ccc}
\left( \begin{array}{cccccc}
0&0&0&\cdots&0&0\\
1&0&0&\cdots&0&0\\
0&1&0&\cdots&0&0\\
\vdots&\vdots&\vdots&\ddots &\vdots&\vdots\\
0&0&0&\cdots&0&0\\
0&0&0&\cdots&1&0
\end{array}\right)
\end{array}.
\]
Then $T=\sum\limits_{j=0}^{b_N}B^j\otimes A_{-j}+
\sum\limits_{j=1}^{b_N}F^j\otimes A_{j}$. Let $e^i_1$ be the $i$th
unit vector in ${\mathbb R}^N$ and $e^j_2$ be the $j$th unit vector
in ${\mathbb R}^m$, then we have
\[
B^l\otimes A_{-l}(e^i_1\otimes
e^j_2)=(B^le^i_1)\otimes(A_{-l}e^j_2)=I_{[1,N]}(i-l)e^{i-l}_1\otimes(A_{-l}e^j_2)
\]
and
\[
F^l\otimes A_{l}(e^i_1\otimes
e^j_2)=(F^le^i_1)\otimes(A_{l}e^j_2)=I_{[1,N]}(i+l)e^{i+l}_1\otimes(A_{l}e^j_2)
\]
thus
\begin{eqnarray*}
T(e^i_1\otimes e^j_2)&=&\sum\limits_{l=0}^{b_N}I_{[1,N]}(i-l)e^{i-l}_1\otimes(A_{-l}e^j_2)+\sum\limits_{l=1}^{b_N}I_{[1,N]}(i+l)e^{i+l}_1\otimes(A_{l}e^j_2)\\
        &=&\sum\limits_{l=-b_N}^{b_N}I_{[1,N]}(i+l)e^{i+l}_1\otimes(A_{l}e^j_2),
\end{eqnarray*}
so
\[
T^k(e^i_1\otimes
e^j_2)=\sum\limits_{l_1,...,l_k=-b_N}^{b_N}\prod\limits_{r=1}^{k}I_{[1,N]}(i+\sum\limits_{q=1}^{r}l_q)e^{i+\sum\limits_{q=1}^{r}l_q}_1\otimes(A_{l_k}\cdots
A_{l_1}e^j_2).
\]
As $\big\{e^i_1\otimes e^j_2\,\big|\,1\le i\le N;1\le j\le m\big\}$
is a standard basis of $\mathbb{R}^{mN}$, we have the trace formula
\begin{eqnarray*}
& &\textrm{tr}(T^k)\\
&=&\sum\limits_{i=1}^{N}\sum\limits_{j=1}^{m}{(e^i_1\otimes e^j_2)}^TT^k(e^i_1\otimes e^j_2)\\
&=&\sum\limits_{i=1}^{N}\sum\limits_{l_1,...,l_k=-b_N}^{b_N}\prod\limits_{r=1}^{k}I_{[1,N]}(i+\sum\limits_{q=1}^{r}l_q)\sum\limits_{j=1}^{m}({(e^i_1)}^T\otimes
{(e^j_2)}^T)\cdot(e^{i+\sum\limits_{q=1}^{r}l_q}_1\otimes(A_{l_k}\cdots A_{l_1}e^j_2))\\
&=&\sum\limits_{i=1}^{N}\sum\limits_{l_1,...,l_k=-b_N}^{b_N}\prod\limits_{r=1}^{k}I_{[1,N]}(i+\sum\limits_{q=1}^{r}l_q)\delta
_{0,\sum\limits_{q=1}^{k}l_q}\cdot\sum\limits_{j=1}^{m}({(e^j_2)}^TA_{l_k}\cdots A_{l_1}e^j_2)\\
&=&\sum\limits_{i=1}^{N}\sum\limits_{l_1,...,l_k=-b_N}^{b_N}\prod\limits_{r=1}^{k}I_{[1,N]}(i+\sum\limits_{q=1}^{r}l_q)\textrm{tr}(A_{l_k}\cdots
A_{l_1})\delta _{0,\sum\limits_{q=1}^{k}l_q}.
\end{eqnarray*}
For $\textrm{tr}(A_{l_k}\cdots A_{l_1})=\textrm{tr}(A_{l_1}\cdots
A_{l_k})$, we get
\begin{eqnarray*}
\textrm{tr}(T^k)=\sum_{i=1}^N
\sum_{j_1,...,j_k=-b_N}^{b_N}\textrm{tr}(A_{j_1}\cdots
A_{j_k})\prod_{l=1}^k I_{[1,N]}(i+\sum_{q=1}^l
j_q)\delta_{0,\sum\limits_{q=1}^k j_q}.\\
\end{eqnarray*}

Note that
\[
\textrm{tr}(A_{l_1}\cdots A_{l_k})=\sum\limits_{t_1,...,t_k=1}^m
a_{t_1t_2}(l_1)\cdots a_{t_kt_1}(l_k),
\]
and we directly get
$$
\textrm{tr}(T^k)=\sum_{i=1}^N
\sum_{j_1,...,j_k=-b_N}^{b_N}\sum_{t_1,...,t_k=1}^m
a_{t_1t_2}(j_1)\cdots a_{t_kt_1}(j_k)\prod_{l=1}^k
I_{[1,N]}(i+\sum_{q=1}^l j_q)\delta_{0,\sum\limits_{q=1}^k j_q}.
$$
\end{proof}

\begin{proof}[Proof of Lemma 2.3]

$\Phi=P\otimes I_m$ where $I_m$ is the $m\times m$ unit matrix and
$P$ is an $N\times N$ matrix and
\[
P=
\begin{array}{ccc}
\left( \begin{array}{cccccc}
0&\cdots&0&1\\
0&\cdots&1&0\\
\vdots&\ddots &\vdots&\vdots\\
1&\cdots&0&0
\end{array}\right)
\end{array}.
\]
So if $A$ is an $N\times N$ matrix and $B$ is an $m\times m$ matrix,
then we have $\Phi(A\otimes B)=(P\otimes I_m)(A\otimes
B)=(PA)\otimes B$. Note that $Pe^i_1=e^{N+1-i}_1$.

As in the proof of Lemma 2.2, we have
\[
\Phi T(e^i_1\otimes e^j_2)=(P\otimes
I_m)(\sum\limits_{l=-b_N}^{b_N}I_{[1,N]}(i+l)e^{i+l}_1\otimes(A^{l}e^j_2))\\
=\sum\limits_{l=-b_N}^{b_N}I_{[1,N]}(i+l)e^{N+1-(i+l)}_1\otimes(A^{l}e^j_2),
\]
and then
\[
H^k(e^i_1\otimes e^j_2)=
\]
\[
\begin{cases}
\sum\limits_{j_1,...,j_k=-b_N}^{b_N}\prod\limits_{r=1}^k
I_{[1,N]}(i-\sum\limits_{q=1}^r
(-1)^q l_q)e_1^{i-\sum\limits_{q=1}^k (-1)^q j_q}\otimes (A_{l_k}\cdots A_{l_1})e^j_2&\text{$k$ even},\\
\sum\limits_{j_1,...,j_k=-b_N}^{b_N}\prod\limits_{r=1}^k
I_{[1,N]}(i-\sum\limits_{q=1}^r
(-1)^q l_q)e_1^{N+1-(i-\sum\limits_{q=1}^k (-1)^q j_q)}\otimes (A_{l_k}\cdots A_{l_1})e^j_2&\text{$k$ odd}.\\
\end{cases}
\]

Thus we get
\begin{eqnarray*}
& &\textrm{tr}(H^k)\\
&=&\sum\limits_{i=1}^{N}\sum\limits_{j=1}^{m}{(e^i_1\otimes e^j_2)}^TH^k(e^i_1\otimes e^j_2)\\
&=&
\begin{cases} \sum\limits_{i=1}^N
\sum\limits_{j_1,...,j_k=-b_N}^{b_N}\textrm{tr}(A_{j_1}\cdots
A_{j_k})\prod\limits_{l=1}^k I_{[1,N]}(i-\sum\limits_{q=1}^l
(-1)^q j_q)\delta_{0,\sum\limits_{q=1}^k (-1)^q j_q}&\text{$k$ even}\\
\sum\limits_{i=1}^N\sum\limits_{j_1,...,j_k=-b_N}^{b_N}\textrm{tr}(A_{j_1}\cdots
A_{j_k})\prod\limits_{l=1}^k I_{[1,N]}(i-\sum\limits_{q=1}^l (-1)^q
j_q)\delta_{2i-1-N,\sum\limits_{q=1}^k (-1)^q j_q}&\text{$k$ odd}
\end{cases}
\end{eqnarray*}
where
\[
 \textrm{tr}(A_{j_1}\cdots A_{j_k})=\sum\limits_{t_1,...,t_k=1}^m
a_{t_1t_2}(j_1)\cdots a_{t_kt_1}(j_k).
\]
\end{proof}

\section{Limit Distributions of Random Block Toeplitz and Hankel Matrices}

Let $\sharp V$ denote the number of elements in an arbitrary finite
set $V$. We now review the concept of partition (see [15]). Let $[n]=\{1,2,...,n\}$. \\
(1) We call $\pi =\{V_1,V_2,...,V_r\}$ a partition of $[n]$ if
$\bigcup\limits_{j=1}^rV_j=[n]$ and $V_i\bigcap V_j=\emptyset$ if
$i\ne j$. \\
(2) For $\forall p\in [n]$, we define $\pi (p)=i$ if $p\in V_i$. We write $p\sim_\pi q$ if $\pi(p)=\pi(q)$. \\
(3) Let $\mathcal{P}(n)$ denote the set of all the partition of
$[n]$. We define $\mathcal{P}_2(n)=\big\{\pi=\{V_1,V_2,...,V_r\}\in
\mathcal{P}(n)\,\big|\,\sharp V_i=2,\forall i\big\}$ and
$\mathcal{P}_2^1(n)=\big\{\pi=\{V_1,V_2,...,V_r\}\in
\mathcal{P}_2(n)\,\big|\, V_i\text{ has exactly one even element and
one odd element},\forall i\big\}.$

\begin{definition}
Let $k\in
\mathbb{N},\,\pi=\{V_1,V_2,...,V_k\}\in\mathcal{P}_2(2k),\,V_r=\{a_r,b_r\}\,(1\le
r\le k)$, then $\pi$ determines a system of linear equations with
unknown variables $t_1,t_2,...,t_{2k}$ (set $t_{2k+1}=t_1$):

\setcounter{equation}{0}
\be
\begin{cases} t_{a_i}=t_{b_i+1}&1\le i\le k\\
t_{b_i}=t_{a_i+1}&1\le i\le k
\end{cases}.
\ee The number of linearly independent equations of this system is
denoted by $f(\pi)$. Clearly we have $0\le f(\pi)\le 2k$.
\end{definition}

\begin{theorem}
Let $T_N=(A_{i-j})_{i,j=1}^N$ be an $mN\times mN$ random block
Toeplitz matrix where $A_s=(a_{ij}(s))^m_{i,j=1}$, $A_{-s}=(A_s)^T$,
and $a_{ij}(s)$'s are random variables satisfying the four
conditions mentioned in introduction. Let $X_N=\frac{T_N}{\sqrt
{mN}}$, then eigenvalue distribution $\mu_{_{X_N}}$ converges almost
surely to a symmetric probability distribution $\gamma_{_T}^{(m)}$
which is determined by its even moments
\[
m_{2k}(\gamma_{_T}^{(m)})=\sum\limits_{\pi\in\mathcal{P}_2(2k)}m^{k-1-f(\pi)}\int_{[0,1]\times
[-1,1]^{k}}\prod_{j=1}^{2k}I_{[0,1]}(x_{0}+\sum_{q=1}^{j}\epsilon_{\pi}(q)\,x_{\pi(q)})
\prod_{l=0}^{k}\mathrm{d}\, x_{l}
\]
where

\be
\epsilon_{\pi}(q)=
\begin{cases} 1&\text{if $q$ is the smaller element of $V_{\pi(q)}$},\\
-1&\text{if $q$ is the larger element of $V_{\pi(q)}$}.
\end{cases}
\ee
\end{theorem}

\begin{proof}[Proof of Theorem 3.2]
The $k$th moment of $\mu_{_{X_N}}$ is given by
\[
m_{k,N}=\int
x^k\mathrm{d}\mu_{_{X_N}}=\frac{1}{mN}E(\textrm{tr}X^k_N)=\frac{1}{mN}(mN)^{-\frac{k}{2}}E(\textrm{tr}T^k_N).
\]
Using Lemma 2.2, we get
\[
m_{k,N}=\frac{1}{mN}(mN)^{-\frac{k}{2}}\sum\limits_{i=1}^N\sum\limits_{j_1,...,j_k=-b_N}^{b_N}\sum\limits_{t_1,...,t_k=1}^mE(a_{t_1t_2}(j_1)\cdots
a_{t_kt_1}(j_k))\prod_{l=1}^k I_{[1,N]}(i+\sum\limits_{q=1}^l
j_q)\delta_{0,\sum\limits_{q=1}^k j_q}.
\]\\
We will calculate the limit $\lim\limits_{N\ra\iy}m_{k,N}$.\\
Observe that
\[
|E(a_{t_1t_2}(j_1)\cdots a_{t_{k}t_1}(j_k))|\le D_k\qquad(\forall
j_1,...,j_k,t_1,...,t_k),
\]
where $D_k=(\max\{C_1,...,C_k\})^k$. If $|j_p|\ne|j_q|$, then $
a_{t_pt_{p+1}}(j_p)$ and $a_{t_qt_{q+1}}(j_q)$ are independent. Let
$p\in\{1,2,...,k\}$. By the independence conditions and
$E(a_{ij}(s))=0$, we observe that if $E(a_{t_1t_2}(j_1)\cdots
a_{t_kt_1}(j_k))\ne0$, then for any $p\in\{1,...,k\},\,\exists
q\in\{1,...,k\},\,q\ne p,\,\textrm{s.t. }|j_q|=|j_p|$. So there are
at most $[\frac{k}{2}]$ different elements in
$V=\{|j_1|,...,|j_k|\}$. Further, we have
\begin{eqnarray*}
\Big|\sum\limits_{j_1,...,j_{k}=-b_N}^{b_N}E(a_{t_1t_2}(j_1)\cdots
a_{t_kt_1}(j_k))\Big|&=&\Big|\sum\limits_{s=1}^{[\frac{k}{2}]}\sum\limits_{\sharp
V=s}E(a_{t_1t_2}(j_1)\cdots a_{t_kt_1}(j_k))\Big|\\
&\le&\sum\limits_{s=1}^{[\frac{k}{2}]}\eta
_{k,s}\frac{(b_N+1)!}{(b_N+1-s)!}2^kD_k
\end{eqnarray*}
where $\eta
_{k,s}=\sharp\Big\{\pi=\{U_1,...,U_s\}\in\mathcal{P}(k)\,\Big|\,\sharp
U_i\ge
2,\forall i\Big\}$. \\
So
$\Big|\sum\limits_{j_1,...,j_{k}=-b_N}^{b_N}E(a_{t_1t_2}(j_1)\cdots
a_{t_kt_1}(j_k))\Big|= O(N^{[\frac{k}{2}]})$ and then $m_{k,N}=
O(N^{[\frac{k}{2}]-\frac{k}{2}})$. Thus for odd $k$, $m_{k,N}=o(1)$.
Now we only have to consider $m_{2k,N}$.

$\pi$ is a partition of $[2k]=\{1,2,...,2k\}$ and $p\sim_\pi q\iff
|j_p|=|j_q|$, then we have

\begin{eqnarray}
& &m_{2k,N}\nonumber\\
&=&\sum\limits_{i=1}^N\sum\limits_{\pi\in\mathcal{P}(2k)}\sum\limits_{j_1,...,j_{2k}=-b_N\atop
p\sim_\pi q\iff
|j_p|=|j_q|}^{b_N}\sum\limits_{t_1,...,t_{2k}=1}^m\frac{E(a_{t_1t_2}(j_1)\cdots
a_{t_{2k}t_1}(j_{2k}))}{(mN)^{k+1}}\prod_{l=1}^{2k}
I_{[1,N]}(i+\sum\limits_{q=1}^l j_q)\delta_{0,\sum\limits_{q=1}^{2k}
j_q}\nonumber\\
&+&o(1).
\end{eqnarray}
Again by the assumptions in the introduction, the contribution of
the partitions which are not pair partition to $m_{2k,N}$ is $o(1)$.
So we only have to consider the pair partitions. Suppose
$\pi\in\mathcal{P}_2(2k)$. If  $p\sim_\pi q$, then $j_p=j_q$ or
$j_p=-j_q$. Under the condition $\sum\limits_{q=1}^{2k}j_q=0$
according to (3.3), considering the main contribution to the trace,
we should take $j_p=-j_q$. Otherwise there exists $p_{0},q_{0}\in
[2k]$ such that
$$j_{p_{0}}=j_{q_{0}}=\frac{1}{2}(j_{p_{0}}+j_{q_{0}}-\sum\limits_{q=1}^{2k}j_{q}).$$
We can choose other $k-1$ distinct numbers, which determine
$j_{p_{0}}=j_{q_{0}}$ and then there is a loss of at least one
degree of freedom and the contribution of such terms is $O(N^{-1})$.
Therefore we have
\begin{eqnarray*}
& &m_{2k,N}\\
&=&\sum\limits_{i=1}^N\sum\limits_{\pi\in\mathcal{P}_2(2k)}\sum\limits_{j_1,...,j_{2k}=-b_N\atop
p\sim_\pi q\iff
j_p=-j_q}^{b_N}\sum\limits_{t_1,...,t_{2k}=1}^m\frac{E(a_{t_1t_2}(j_1)\cdots
a_{t_{2k}t_1}(j_{2k}))}{(mN)^{k+1}}\prod_{l=1}^{2k}
I_{[1,N]}(i+\sum\limits_{q=1}^l j_q)\\
&+&o(1)\\
&=&\sum\limits_{i=1}^N\sum\limits_{\pi\in\mathcal{P}_2(2k)}\sum\limits_{\mbox{\tiny
$\begin{array}{c}j_1,...,j_{2k}=-b_N\\
j_t\ne0,\forall t\\p\sim_\pi q\iff
j_p=-j_q\end{array}$}}^{b_N}\sum\limits_{t_1,...,t_{2k}=1}^m\frac{E(a_{t_1t_2}(j_1)\cdots
a_{t_{2k}t_1}(j_{2k}))}{(mN)^{k+1}}\prod_{l=1}^{2k}
I_{[1,N]}(i+\sum\limits_{q=1}^l j_q)\\
&+&o(1) .
\end{eqnarray*}

Now we have to calculate
\[
\sum\limits_{t_1,...,t_{2k}=1}^mE(a_{t_1t_2}(j_1)\cdots
a_{t_{2k}t_1}(j_{2k}))
\]
with the assumption that $j_t\ne0,\,\forall t$. For convenience we
let $t_{2k+1}=t_1$. Suppose
\[
E(a_{t_1t_2}(j_1)\cdots a_{t_{2k}t_1}(j_{2k}))\ne0,
\]
then for any $p\sim_\pi q$ we have
$a_{t_pt_{p+1}}(j_p)=a_{t_qt_{q+1}}(j_q)$. Suppose $p<q$.
\begin{eqnarray*}
&&E(a_{t_1t_2}(j_1)\cdots a_{t_{2k}t_1}(j_{2k}))\\
&=&E(a_{t_pt_{p+1}}(j_p)a_{t_qt_{q+1}}(j_q))\cdot E(a_{t_1t_2}(j_1)\cdots \widehat{a_{t_pt_{p+1}}(j_p)}\cdots \widehat{a_{t_qt_{q+1}}(j_q)}\cdots a_{t_{2k}t_1}(j_{2k}))\\
&=&1\cdot E(a_{t_1t_2}(j_1)\cdots
\widehat{a_{t_pt_{p+1}}(j_p)}\cdots
\widehat{a_{t_qt_{q+1}}(j_q)}\cdots a_{t_{2k}t_1}(j_{2k}))
\end{eqnarray*}
where $\,\widehat{}\,$ marks omitted index. Note that
$j_p=-j_q,\,A_{j_p}=(A_{j_q})^T$ and $A_{j_p}\ne A_{j_q}$ because
$j_p=-j_q\ne0$. $a_{t_pt_{p+1}}(j_p)$ lies in the $t_p$-th row and
the $t_{p+1}$-th column of $A_{j_p}$ and $a_{t_qt_{q+1}}(j_q)$ lies
in the $t_q$-th row and the $t_{q+1}$-th column of $A_{j_q}$. So
$t_p=t_{q+1},\,t_q=t_{p+1}$. As $\pi=\{V_1,V_2,...,V_r\}$ where
$V_i=\{a_i,b_i\}\,(1\le i\le r)$, we have a system of equations:
\[
\begin{cases}t_{a_1}=t_{b_1+1}\\
t_{a_1+1}=t_{b_1}\\
\cdots \\
t_{a_k}=t_{b_k+1}\\
t_{a_k+1}=t_{b_k}\\
\end{cases}.
\]
By Definition 3.1, there are $f(\pi)$ linearly independent equations
in this system. So there are $2k-f(\pi)$ variables taking values
freely in $\{1,2,...,m\}$ and the number of solutions of this system
of equations is $m^{2k-f(\pi)}$. In other words, there are
$m^{2k-f(\pi)}$ different $(t_1,...,t_{2k})$'s such that
\[
E(a_{t_1t_2}(j_1)\cdots a_{t_{2k}t_1}(j_{2k}))\ne0
\]
and this implies
\[
E(a_{t_1t_2}(j_1)\cdots a_{t_{2k}t_1}(j_{2k}))=1.
\]
So
\[
\sum\limits_{t_1,...,t_{2k}=1}^mE(a_{t_1t_2}(j_1)\cdots
a_{t_{2k}t_1}(j_{2k}))=m^{2k-f(\pi)}.
\]

Thus we have
\begin{eqnarray*}
& &m_{2k,N}\\
&=&\frac{1}{(mN)^{k+1}}\sum\limits_{i=1}^N\sum\limits_{\pi\in\mathcal{P}_2(2k)}m^{2k-f(\pi)}\sum\limits_{\mbox{\tiny
$\begin{array}{c}j_1,...,j_{2k}=-b_N\\
j_t\ne0,\forall t\\p\sim_\pi q\iff
j_p=-j_q\end{array}$}}^{b_N}\prod_{l=1}^{2k}
I_{[1,N]}(i+\sum\limits_{q=1}^l j_q)\\
&+&o(1)\\
&=&\frac{1}{(mN)^{k+1}}\sum\limits_{i=1}^N\sum\limits_{\pi\in\mathcal{P}_2(2k)}m^{2k-f(\pi)}\sum\limits_{j_1,...,j_{2k}=-b_N\atop
p\sim_\pi q\iff j_p=-j_q}^{b_N}\prod_{l=1}^{2k}
I_{[1,N]}(i+\sum\limits_{q=1}^l j_q)\\
&+&o(1).
\end{eqnarray*}

Now for any $r\in \{1,...,k\}$, let $x_r=j_{a_r}$, then $ j_q=
\begin{cases}x_r &\text{if $q=a_r$}\\
-x_r &\text{if $q=b_r$}
\end{cases}
$. Remember Eq.(3.2), and then we have $j_q=
\begin{cases}x_{\pi(q)} &\text{if $\epsilon_{\pi}(q)=1$}\\
-x_{\pi(q)} &\text{if $\epsilon_{\pi}(q)=-1$}
\end{cases}$, so $j_q=\epsilon_{\pi}(q)x_{\pi(q)}$ and
\[
m_{2k,N}=(mN)^{-k-1}\sum\limits_{i=1}^N\sum\limits_{\pi\in\mathcal{P}_2(2k)}\sum\limits_{x_1,...,x_k=-b_N}
^{b_N}m^{2k-f(\pi)}\prod_{l=1}^{2k} I_{[1,N]}(i+\sum_{q=1}^l
\epsilon_{\pi}(q)x_{\pi(q)})+o(1),\\
\]
thus
\[
\lim_{N\ra\iy}m_{2k,N}=\sum\limits_{\pi\in\mathcal{P}_2(2k)}m^{k-1-f(\pi)}\int_{[0,1]\times
[-1,1]^{k}}\prod_{j=1}^{2k}I_{[0,1]}(x_{0}+\sum_{q=1}^{j}\epsilon_{\pi}(q)\,x_{\pi(q)})
\prod_{l=0}^{k}\mathrm{d}\, x_{l}.
\]

Let
$\displaystyle{m_{2k}=\sum\limits_{\pi\in\mathcal{P}_2(2k)}m^{k-1-f(\pi)}\int_{[0,1]\times
[-1,1]^{k}}\prod\limits_{j=1}^{2k}I_{[0,1]}(x_{0}+\sum\limits_{q=1}^{l}\epsilon_{\pi}(q)\,x_{\pi(q)})
\prod\limits_{l=0}^{k}\mathrm{d}\, x_{l}}$ and
$m_{2k-1}=0(k\in\mathbb{N})$, then for any $k\in\mathbb{N}$ we have
$\displaystyle{\lim\limits_{N\ra\iy}m_{k,N}=m_k}$. It is easy to see
that
$m_{2k}\le\sum\limits_{\pi\in\mathcal{P}_2(2k)}m^{k-1-f(\pi)}\le(2k-1)!!\cdot
m^{k-1}$ and then using Carleman's theorem (see [6]) we know that
the limit distribution $\gamma_{_T}^{(m)}$ is uniquely determined by
its moments $\{m_k\}_{k=0}^{\iy}$.

Now we prove the almost sure convergence. It is sufficient to prove

\be
\sum\limits_{N=1}^{\iy}\frac{1}{N^4}E\Big(\big(\textrm{tr}X_N^k-E(\textrm{tr}X_N^k)\big)^4\Big)<\iy.
\ee
\begin{align*}
&\textrm{tr}X_N^k=(mN)^{-\frac{k}{2}}\textrm{tr}T_N^k\\
&=(mN)^{-\frac{k}{2}}\sum\limits_{i=1}^N\sum\limits_{t_1,...,t_k=1}^{m}\sum\limits_{j_1,...,j_k=-b_N}^{b_N}a_{t_1t_2}(j_1)\cdots
a_{t_{k}t_1}(j_{k})\prod_{l=1}^{k}
I_{[1,N]}(i+\sum\limits_{q=1}^lj_q)\delta_{0,\sum\limits_{q=1}^{k}j_q}.\\
\end{align*}
For convenience we let
\begin{eqnarray*}
\textbf j=(j_1,...,j_k),\quad \textbf t=(t_1,...,t_k),\\
A[i,\textbf t,\textbf j]=a_{t_1t_2}(j_1)\cdots
a_{t_kt_1}(j_{k})\prod_{l=1}^{k}
I_{[1,N]}(i+\sum\limits_{q=1}^lj_q)\delta_{0,\sum\limits_{q=1}^{k}j_q},
\end{eqnarray*}
and we use $\sum\limits_{i,\textbf t,\textbf j}$ to denote
$\sum\limits_{i=1}^N\sum\limits_{t_1,...,t_k=1}^{m}\sum\limits_{j_1,...,j_k=-b_N}^{b_N}$,
then we have
\begin{eqnarray*}
& &\frac{1}{N^4}E\Big(\big(\textrm{tr}X_N^k-E(\textrm{tr}X_N^k)\big)^4\Big)\\
&=&\frac{1}{N^4}E\Big((mN)^{-\frac{k}{2}}\sum\limits_{i,\textbf t,\textbf j}A[i,\textbf t,\textbf j]-E\big((mN)^{-\frac{k}{2}}\sum\limits_{i,\textbf t,\textbf j}A[i,\textbf t,\textbf j]\big)\Big)^4\\
&=&\frac{1}{N^{2k+4}m^{2k}}E\big(\sum\limits_{i,\textbf t,\textbf j}(A[i,\textbf t,\textbf j]-EA[i,\textbf t,\textbf j])\big)^4\\
&=&\frac{1}{N^{2k+4}m^{2k}}E\Big(\sum\limits_{v=1}^4\sum\limits_{i^v,\textbf
t^v,\textbf j^v}\prod_{v=1}^4(A[i^v,\textbf t^v,\textbf
j^v]-EA[i^v,\textbf t^v,\textbf j^v])\Big)
\end{eqnarray*}
where $\sum\limits_{i^v,\textbf t^v,\textbf
j^v}=\sum\limits_{i^v=1}^N\sum\limits_{t_1^v,...,t_k^v=1}^m\sum\limits_{j^v_1,...,j^v_k=-b_N}^{b_N}$.

For given $\textbf j^1=(j_1^1,...,j_k^1)$, $\textbf
j^2=(j_1^2,...,j_k^2)$, $\textbf j^3=(j_1^3,...,j_k^3)$, $\textbf
j^4=(j_1^4,...,j_k^4)$, let
$$\textbf
J=(j_1^1,...,j_k^1,...,j_1^4,...,j_k^4)\in\{-b_N,...,b_N\}^{4k}.$$
Set $S^\textbf J=\{|j_1^1|,...,|j_k^1|,...,|j_1^4|,...,|j_k^4|\}$.
We use $p(\textbf J)$ to denote the number of different elements of
$S^\textbf J$.

We construct a set of numbers with multiplicities $S_\textbf
J=\{j_1^1,...,j_k^1,...,j_1^4,...,j_k^4\}$. Please note that if
$|j_{u_1}^{v_1}|=|j_{u_2}^{v_2}|=a$, then $a$ appears twice in
$S_\textbf J$. Let $S_1,...,S_{p(\textbf J)}$
be subsets of $S_{\textbf J}$ such that:\\
(a) for all $w$, the elements of $S_w$ have the same absolute
value;\\
(b) if $w_1\ne w_2$, then the absolute value of the elements in
$S_{w_1}$ is different from the absolute value of the elements in
$S_{w_2}$;\\
(c) $\bigcup\limits_{r=1}^pS_r=S_{\textbf J}$. It is easy to see
that $S_1,...,S_{p(\textbf J)}$ are uniquely determined by $\textbf
J$. Then we have
\[
\{-b_N,...,b_N\}^{4k}=W_1\cup W_2\cup W_3
\]
where
\[
W_1=\{\textbf J\in\{-b_N,...,b_N\}^{4k}\big|p(\textbf J)\le2k-2\},
\]
\[
W_2=\{\textbf J\in\{-b_N,...,b_N\}^{4k}\big|p(\textbf J)=2k-1\},
\]
\[
W_3=\{\textbf J\in\{-b_N,...,b_N\}^{4k}\big|p(\textbf J)=2k\}.
\]

Therefore we have
\begin{eqnarray*}
& &\frac{1}{N^4}E\Big(\big(\textrm{tr}X_N^k-E(\textrm{tr}X_N^k)\big)^4\Big)\\
&=&\frac{1}{N^{2k+4}m^{2k}}E\Big(\sum\limits_{\mbox{\tiny
$\begin{array}{c}i^1,i^2,i^3,i^4\\
\textbf t^1,\textbf t^2,\textbf t^3,\textbf
t^4\end{array}$}}\sum\limits_{\textbf J\in
W_1}\prod_{v=1}^4(A[i^v,\textbf t^v,\textbf j^v]-EA[i^v,\textbf
t^v,\textbf j^v])\Big)\\
&+&\frac{1}{N^{2k+4}m^{2k}}E\Big(\sum\limits_{\mbox{\tiny
$\begin{array}{c}i^1,i^2,i^3,i^4\\
\textbf t^1,\textbf t^2,\textbf t^3,\textbf
t^4\end{array}$}}\sum\limits_{\textbf J\in
W_2}\prod_{v=1}^4(A[i^v,\textbf t^v,\textbf j^v]-EA[i^v,\textbf
t^v,\textbf j^v])\Big)\\
&+&\frac{1}{N^{2k+4}m^{2k}}E\Big(\sum\limits_{\mbox{\tiny
$\begin{array}{c}i^1,i^2,i^3,i^4\\
\textbf t^1,\textbf t^2,\textbf t^3,\textbf
t^4\end{array}$}}\sum\limits_{\textbf J\in
W_3}\prod_{v=1}^4(A[i^v,\textbf t^v,\textbf j^v]-EA[i^v,\textbf
t^v,\textbf j^v])\Big)\\
&=&\Phi_1+\Phi_2+\Phi_3.
\end{eqnarray*}

For given $j_1^1,...,j_k^4,t_1^1,...,t_k^4,i^1,...,i^4$, suppose
$E(\prod\limits_{v=1}^4(A[i^v,\textbf t^v,\textbf
j^v]-EA[i^v,\textbf t^v,\textbf j^v]))\ne0$. Then
$\sum\limits_{i=1}^kj_i^v=0\,(1\le v\le4)$ and from
independence conditions we know that:\\
\begin{eqnarray}
\forall j_u^v,\,\exists (u_1,v_1)\ne(u,v)\text{ s.t.
}|j_u^v|=|j_{u_1}^{v_1}|
\end{eqnarray}
and that\\
\begin{eqnarray}
\forall v_1,\,\exists v_2\ne v_1\text{ and }u_1,u_2\text{ s.t.
}|j_{u_1}^{v_1}|=|j_{u_2}^{v_2}|;
\end{eqnarray}
otherwise,\\
\begin{eqnarray*}
& &E\big(\prod_{v=1}^4(A[i^v,\textbf t^v,\textbf j^v]-EA[i^v,\textbf
t^v,\textbf j^v])\big)\\
&=&E\big((A[i^{v_1},\textbf t^{v_1},\textbf
j^{v_1}]-EA[i^{v_1},\textbf t^{v_1},\textbf j^{v_1}])\big)\cdot
E\big(\prod\limits_{\mbox{\tiny
$\begin{array}{c}{1\le v\le4}\\
v\ne v_1\end{array}$}}(A[i^v,\textbf t^v,\textbf j^v]-EA[i^v,\textbf t^v,\textbf j^v])\big)\\
&=&0.
\end{eqnarray*}
We now evaluate $\Phi_1$, $\Phi_2$ and $\Phi_3$.

\begin{bfseries}
\begin{large}
Evaluation of $\Phi_1$.
\end{large}
\end{bfseries}
Suppose $p\in\mathbb{N}$ and $p\le2k-2$. There are at most
$R_{k,p}\cdot(b_N+1)^p\cdot2^{4k}$ different $\textbf J$'s that
satisfy $p(\textbf J)=p$, where
$R_{k,p}=\sharp\big\{\pi=\{U_1,..,U_p\}\big|\pi\in\mathcal{P}(4k)\big\}$.
So we have
\begin{eqnarray*}
\Phi_1&\le&\sum\limits_{p=1}^{2k-2}\frac{1}{N^{2k+4}m^{2k}}\cdot
N^4\cdot m^{4k}\cdot
R_{k,p}\cdot (b_N+1)^p\cdot2^{4k}\cdot M\\
&\le&C\cdot b_N^{-2}
\end{eqnarray*} where

\begin{eqnarray*}
M=\sup\limits_{N}\Big\{\max\limits_{\mbox{\tiny
$\begin{array}{c}1\le i^v\le N\\\textbf
J\in\{-b_N,...,b_N\}^{4k}\\\textbf
t^v\in\{1,...,m\}^k\end{array}$}}\Big|E[\prod\limits_{v=1}^4(A[i^v,\textbf
t^v,\textbf j^v]-EA[i^v,\textbf t^v,\textbf j^v])]\Big|\Big\}
\end{eqnarray*}
and $C$ is independent of $N$. From (1.1), (1.2) and (1.3), we know that such $M$ exists. \\

\begin{bfseries}
\begin{large}
Evaluation of $\Phi_2$.
\end{large}
\end{bfseries}
Suppose $\textbf J\in W_2$. For any $v\in\{1,2,3,4\}$, if $\exists
u\in\{1,2,...,k\}$ such that $\forall u_1\ne
u,\,|j_u^v|\ne|j_{u_1}^v|$, then $j_u^v$ and its absolute value are
determined by $j_1^v,...,j_{u-1}^v,j_{u+1}^v,...,j_k^v$ for
$\sum\limits_{i=1}^kj_i^v=0$. Thus
\begin{eqnarray}
& &\sharp\Big\{\textbf J=(j_1^1,...,j_k^4)\in W_2\Big|\forall
v,\sum\limits_{l=1}^kj_l^v=0\text{ and }\exists
j_{u_1}^{v_1}\,\textrm{s.t. }\forall
u_2\ne u_1,\,|j_{u_1}^{v_1}|\ne|j_{u_2}^{v_1}|\Big\}\nonumber\\
&\le&Q_k\cdot(b_N+1)^{2k-2}\cdot2^{4k}.
\end{eqnarray}
where $Q_k=\sharp\big\{\pi=\{V_1,...,V_{2k-1}\}\big|\pi\in
\mathcal{P}(4k)\big\}$. When $\textbf J\in W_2$, there are two
situations. One is that $\exists i_1,i_2$ such that $\sharp
S_{i_1}=\sharp S_{i_2}=3$ and for any $i\notin\{i_1,i_2\},\,\sharp
S_i=2$. The other situation is that $\exists i_1$ such that $\sharp
S_{i_1}=4$ and for any $i\ne i_1,\,\sharp S_i=2$. In the first
situation, suppose $S_{i_1}=\{j_x^a,j_y^b,j_z^c\}$,
$S_{i_2}=\{j_w^d,j_u^e,j_v^f\}$. If $a=b=c$ and $d=e=f$, then we can
find $g\in\{1,2,3,4\}\backslash \{a,d\}$. Then for any $r$, there is
one unique element which has the same absolute value as $j_r^g$ in
$S_\textbf J$. Then from (3.6) we know that in $\{j_1^g,...,j_k^g\}$
there is at least one element which has a different absolute value
from the others and then (3.7) is obtained. If $a\ne b$ and $a\ne
c$, then $j_x^a$ has a different absolute value from the other
elements in $\{j_1^a,...,j_k^a\}$ and then (3.7) is also obtained.
In the second situation, suppose
$S_{i_1}=\{j_x^a,j_y^b,j_z^c,j_w^d\}$. If
$\{a,b,c,d\}\ne\{1,2,3,4\}$ then similarly to the first situation we
can get (3.7). If $\{a,b,c,d\}=\{1,2,3,4\}$, then $j_x^a$ has
different absolute value from the other elements in
$\{j_1^a,...,j_k^a\}$ thus we also get (3.7). From the above
discussion we know that for any $\textbf J\in W_2$ such that
$E(\prod\limits_{v=1}^4(A[i^v,\textbf t^v,\textbf
j^v]-EA[i^v,\textbf t^v,\textbf j^v]))\ne0$, (3.7) is satisfied.
Then we get
\begin{eqnarray*}
& &\Phi_2\\
&=&\frac{1}{N^{2k+4}m^{2k}}E(\sum\limits_{\mbox{\tiny
$\begin{array}{c}i^1,i^2,i^3,i^4\\
\textbf t^1,\textbf t^2,\textbf t^3,\textbf
t^4\end{array}$}}\sum\limits_{\textbf J\in W_2}E(\prod_{v=1}^4(A[i^v,\textbf t^v,\textbf j^v]-EA[i^v,\textbf t^v,\textbf j^v]))\\
&\le&\frac{1}{N^{2k+4}m^{2k}}N^4m^{4k}Q_k(b_N+1)^{2k-2}2^{4k}M\\
&\le&D\cdot b_N^{-2}
\end{eqnarray*}
where $D$ is independent of $N$.
\\

\begin{bfseries}
\begin{large}
Evaluation of $\Phi_3$.
\end{large}
\end{bfseries}
When $\textbf J\in W_3$, we have $\sharp S_i=2$ for all
$i\in\{1,...,2k\}$. Like the situations mentioned above, we have:
$\exists u_1\in\{1,2,...,k\}$ such that $\forall u_2\ne
u_1,\,|j_{u_1}^1|\ne|j_{u_2}^1|$, then $j_{u_1}^1$ and its absolute
value are determined by
$j_1^1,...,j_{u_1-1}^1,j_{u_1+1}^1,...,j_k^1$ for
$\sum\limits_{i=1}^kj_i^1=0$. Suppose $|j_{u_3}^e|=|j_{u_1}^1|$ and
$f\in\{1,2,3,4\}\backslash\{1,e\}$, then $\exists
u_4\in\{1,2,...,k\}$ such that $\forall u_5\ne
u_4,\,|j_{u_5}^f|\ne|j_{u_4}^f|$, then $j_{u_4}^f$ and its absolute
value are determined by
$j_1^f,...,j_{u_4-1}^f,j_{u_4+1}^f,...,j_k^f$ for
$\sum\limits_{i=1}^kj_i^f=0$. For $j_{u_1}^1$ and $j_{u_4}^f$ are
determined, we have
\begin{eqnarray*}
& &\sharp\Big\{\textbf J=(j_1^1,...,j_k^1,...,j_k^4)\in
W_3\Big|E(\prod\limits_{v=1}^4(A[i^v,\textbf t^v,\textbf
j^v]-EA[i^v,\textbf t^v,\textbf j^v]))\ne0\Big\}\\
&\le&(b_N+1)^{p-2}\cdot2^{4k}\cdot\sharp\mathcal{P}_2(4k)\\
\end{eqnarray*}
and thus
\begin{eqnarray*}
& &\Phi_3\\
&=&\frac{1}{N^{2k+4}m^{2k}}E\Big(\sum\limits_{\mbox{\tiny
$\begin{array}{c}i^1,i^2,i^3,i^4\\
\textbf t^1,\textbf t^2,\textbf t^3,\textbf
t^4\end{array}$}}\sum\limits_{\textbf J\in W_3}E\big(\prod_{v=1}^4(A[i^v,\textbf t^v,\textbf j^v]-EA[i^v,\textbf t^v,\textbf j^v])\big)\Big)\\
&\le&\frac{1}{N^{2k+4}m^{2k}}N^4m^{4k}(b_N+1)^{p-2}\cdot2^{4k}\cdot\sharp\mathcal{P}_2(4k)M\\
&\le&F\cdot b_N^{-2}
\end{eqnarray*}
where $F$ is independent of $N$. Finally,
\begin{eqnarray*}
\frac{1}{N^4}E((\textrm{tr}X_N^k-E(\textrm{tr}X_N^k))^4)=\Phi_1+\Phi_2+\Phi_3\le(C+D+F)\cdot
b_N^{-2}.
\end{eqnarray*}
For $b_N=N-1$, (3.4) is proved and we have proved the theorem.
\end{proof}

\begin{remark}[\textbf{Hankel block matrices and complex Toeplitz case}]
If $H_N=(A_{N+1-i-j})$ is a block Hankel matrix which is defined as
in the introduction, then the eigenvalue distribution of
$H_N/\sqrt{mN}$ converges almost surely to a distribution
$\gamma_H^{(m)}$ which is determined by its even moments
\begin{eqnarray}
&
&m_{2k}(\gamma_{_H}^{(m)})=\sum\limits_{\pi\in\mathcal{P}_2^1(2k)}\frac{r(m,\pi)}{m^{k+1}}\int_{[0,1]\times
[-1,1]^{k}}\prod_{j=1}^{2k}I_{[0,1]}(x_{0}-\sum_{q=1}^{j}(-1)^q\,x_{\pi(q)})
\prod_{l=0}^{k}\mathrm{d}\, x_{l}\nonumber\\
& &
\end{eqnarray}
where
$r(m,\pi)=\sharp\big\{(t_1,..,t_{2k})\,\big|\,E(a_{t_1t_2}(j_1),...,a_{t_{2k}t_1}(j_{2k}))\ne0\big\}$.\\
Suppose $\pi=\{\{a_1,b_1\},...,\{a_k,b_k\}\}$, then
\begin{eqnarray}
& &r(m,\pi)\nonumber\\
&=&\sharp\big\{(t_1,..,t_{2k})\,\big|1\le t_{a_i},t_{b_i}\le m;
\begin{cases}
t_{a_i}=t_{b_i}\\
t_{a_{i+1}}=t_{b_{i+1}}
\end{cases}
\text{or }
\begin{cases}
t_{a_i}=t_{b_{i+1}}\\
t_{a_{i+1}}=t_{b_i}
\end{cases}
(1\le i\le
k)\big\}\nonumber\\
&=&\sharp\big\{(t_1,..,t_{2k})\,\big|\,1\le t_{a_i},t_{b_i}\le
m;\,t_{a_i}=t_{(b_i+1)};\,t_{(a_i+1)}=t_{b_i}\,(1\le i\le k)\big\}\nonumber\\
&+&O(m^k)\nonumber\\
&=&m^{2k-f(\pi)}+O(m^k)
\end{eqnarray}
If the blocks of $H_N$ are Hermitian matrices, the results would be
similar. If the blocks in a block Toeplitz matrix
$T=(A_{i-j})_{i,j=1}^N$ are complex matrices and
$A_{-s}=(\overline{A_s})^T$, then the results would also be similar.
\end{remark}

\begin{remark}[\textbf{relation with dynamical system}] $f(\pi)$ has a
relation with dynamical system. Let
$\pi=\{V_1,V_2,...,V_k\},\,V_r=\{a_r,b_r\}\,(1\le r\le k)$, then
$f(\pi)$ is the number of linearly independent equations of (3.1).
Now we consider a discrete dynamical system. Give
$[2k]=\{1,...,2k\}$ the discrete topology. Let $\phi$ be a
self-homeomorphism on $[2k]$ and
$\phi(a_i)=b_i+1,\,\phi(b_i)=a_i+1\,(1\le i\le k)$. Consider a
continuous map $\psi:\mathbb{Z}\times[2k]\to[2k]$ such that
$\psi(s,t)=\phi^s(t)\,(s\in\mathbb{Z},\,t\in[2k])$ then $\psi$
becomes a dynamical system. Obviously we see that the number of
orbits of this dynamical system equals the number of independent
variables in (3.1) and then equals $2k-f(\pi)$.
\end{remark}

\section{Block Toeplitz and Hankel Band Matrices}

\begin{definition}
Let $T_N=(A_{i-j})_{i,j=1}^{N}$ be an $mN\times mN$ block Toeplitz
matrix. We call $T$ a block Toeplitz band matrix if $\exists\,b_N<N$
s.t. $A_s=0$ when $|s|>b_N$. We call $b_N$ the bandwidth of the
matrix.
\end{definition}

\begin{theorem}[proportional
growth] Let $T_N$ be a block Toeplitz band matrix with the bandwidth
$b_{N}\sim b N,\ b\in (0,1]$. Take the normalization
$$X_N=T_N/\sqrt{m(2-b)bN}.$$ With the notation
and assumptions of Theorem 3.2, $\mu_{_{X_N}}$ converges almost
surely to a symmetric probability distribution
$\gamma_{_{T}}^{(m)}(b)$ which is determined by its even moments

\[
m_{2k}(\gamma_{_{T}}^{(m)}(b))=\frac{1}{(2-b)^{k}}\sum_{\pi\in
\mathcal{P}_{2}(2k)}m^{k-1-f(\pi)}\int_{[0,1]\times
[-1,1]^{k}}\prod_{j=1}^{2k}I_{[0,1]}(x_{0}+b\sum_{i=1}^{j}\epsilon_{\pi}(i)\,x_{\pi(i)})
\prod_{l=0}^{k}\mathrm{d}\, x_{l} \]
 where $\epsilon_{\pi}(i)=1$ if
$i$ is the smaller number of $V_{\pi(i)}$ and $\epsilon_{\pi}(i)=-1$
otherwise.
\end{theorem}

\begin{theorem}[slow growth]
Let $T_N$ be a block Toeplitz band matrices with the bandwidth
$b_N=o(N)$ but $b_N\ra \iy$. Take the normalization
$$X_N=T_N/(\sqrt{2mb_N}).$$ With the notation
and assumptions of Theorem 3.2, $\mu_{_{X_N}}$ converges weakly to a
distribution $\gamma^{(m)}$ which is determined by its even moments
\[
m_{2k}(\gamma^{(m)})=\sum_{\pi\in
\mathcal{P}_{2}(2k)}m^{k-1-f(\pi)}.
\]
In addition, if there exist positive constants $\epsilon_0$ and $C$
such that
\[
b_N\ge C\cdot N^{\frac{1}{2}+\epsilon_0},
\]
then $\mu_{_{X_N}}$
converges almost surely to $\gamma^{(m)}$.
\end{theorem}

\begin{proof}[Proof of Theorem 4.2]
It is easy to see that Lemma 2.2 is also right for block Toeplitz
band matrices but now $b_N$ is no longer $N-1$. Let
\[
m_{k,N}=\int
x^k\mathrm{d}\mu_{_{X_N}}=\frac{1}{mN}E(\textrm{tr}X^k_N)=\frac{1}{mN}(m(2-b)bN)^{-\frac{k}{2}}E(\textrm{tr}T^k_N).
\]
Using Lemma 2.2, we get
\begin{eqnarray*}
& &m_{k,N}\\
&=&\frac{1}{(2-b)^{\frac{k}{2}}}\frac{1}{mN}\sum\limits_{i=1}^N\sum\limits_{j_1,...,j_k=-b_N}^{b_N}\sum\limits_{t_1,...,t_k=1}^m\frac{E(a_{t_1t_2}(j_1)\cdots
a_{t_kt_1}(j_k))}{(mbN)^{\frac{k}{2}}}\prod_{l=1}^k
I_{[1,N]}(i+\sum\limits_{q=1}^l j_q)\delta_{0,\sum\limits_{q=1}^k
j_q}\\
&=&\frac{1}{(2-b)^{\frac{k}{2}}}{(\frac{1}{b})}^{\frac{k}{2}}\sum\limits_{i=1}^N\sum\limits_{j_1,...,j_k=-b_N}^{b_N}\sum\limits_{t_1,...,t_k=1}^m\frac{E(a_{t_1t_2}(j_1)\cdots
a_{t_kt_1}(j_k))}{(mN)^{\frac{k}{2}+1}}\prod_{l=1}^k
I_{[1,N]}(i+\sum\limits_{q=1}^l j_q)\delta_{0,\sum\limits_{q=1}^k
j_q}.\\
\end{eqnarray*}
Similarly as in the proof of Theorem 3.2, we know that when $k$ is
odd,
\[
\sum\limits_{i=1}^N\sum\limits_{j_1,...,j_k=-b_N}^{b_N}\sum\limits_{t_1,...,t_k=1}^m\frac{E(a_{t_1t_2}(j_1)\cdots
a_{t_kt_1}(j_k))}{(mN)^{\frac{k}{2}+1}}\prod_{l=1}^k
I_{[1,N]}(i+\sum\limits_{q=1}^l j_q)\delta_{0,\sum\limits_{q=1}^k
j_q}=o(1),
\]
and then $m_{k,N}=o(1)$.

For $m_{2k,N}$,
\begin{eqnarray*}
&&\sum\limits_{i=1}^N\sum\limits_{j_1,...,j_{2k}=-b_N}^{b_N}\sum\limits_{t_1,...,t_{2k}=1}^m\frac{E(a_{t_1t_2}(j_1)\cdots
a_{t_{2k}t_1}(j_{2k}))}{(mN)^{k+1}}\prod_{l=1}^{2k}
I_{[1,N]}(i+\sum\limits_{q=1}^l j_q)\delta_{0,\sum\limits_{q=1}^{2k}
j_q}\\
&=&N^{-k-1}\sum\limits_{i=1}^N\sum\limits_{\pi\in\mathcal{P}_2(2k)}\sum\limits_{x_1,...,x_k=-b_N}
^{b_N}m^{k-1-f(\pi)}\prod_{l=1}^{2k} I_{[1,N]}(i+\sum_{q=1}^l
\epsilon_{\pi}(q)x_{\pi(q)})+o(1).
\end{eqnarray*}
So
\begin{eqnarray*}
& &m_{2k,N}\\
&=&\frac{1}{(2-b)^k}{(\frac{1}{b})}^kN^{-k-1}\sum\limits_{i=1}^N\sum\limits_{\pi\in\mathcal{P}_2(2k)}\sum\limits_{x_1,...,x_k=-b_N}
^{b_N}m^{k-1-f(\pi)}\prod_{l=1}^{2k} I_{[1,N]}(i+\sum_{q=1}^l
\epsilon_{\pi}(q)x_{\pi(q)})+o(1),
\end{eqnarray*}
thus
\[
\lim_{N\ra\iy}m_{2k,N}=\frac{1}{(2-b)^k}\sum\limits_{\pi\in\mathcal{P}_2(2k)}m^{k-1-f(\pi)}\int_{[0,1]\times
[-1,1]^{k}}\prod_{j=1}^{2k}I_{[0,1]}(x_{0}+b\sum_{q=1}^{j}\epsilon_{\pi}(q)\,x_{\pi(q)})
\prod_{l=0}^{k}\mathrm{d}\, x_{l}.\\
\]
Let $m_{2k}=\lim\limits_{N\ra\iy}m_{2k,N}$ and
$m_{2k-1}=0(k\in\mathbb{N})$, then for any $k\in\mathbb{N}$ we have
$\lim\limits_{N\ra\iy}m_{k,N}=m_k$. It is easy to see that
$$m_{2k}\le\sum\limits_{\pi\in\mathcal{P}_2(2k)}m^{k-1-f(\pi)}\le(2k-1)!!\cdot
m^{k-1}$$ and then using Carleman's theorem (see [6]) we know that
the limit distribution $\gamma_{_T}^{(m)}(b)$ is determined by its
even moments $\{m_{2k}\}$. Similar to the proof of Theorem 3.2 we
can easily prove that $\mu_{_{X_N}}$ converges almost surely to
$\gamma_{_T}^{(m)}(b)$. Then Theorem 4.2 is proved.
\end{proof}

\begin{proof}[Proof of Theorem 4.3]

Let
\begin{eqnarray*}
& &m_{k,N}\\
&=&\int
x^k\mathrm{d}\mu_{_{X_N}}=\frac{1}{mN}E(\textrm{tr}X^k_N)=\frac{1}{mN}(2mb_N)^{-\frac{k}{2}}E(\textrm{tr}T^k_N)\\
&=&\frac{(mN)^{\frac{k}{2}+1}}{mN(2mb_N)^{\frac{k}{2}}}\sum\limits_{i=1}^N\sum\limits_{j_1,...,j_k=-b_N}^{b_N}\sum\limits_{t_1,...,t_k=1}^m\frac{E(a_{t_1t_2}(j_1)\cdots
a_{t_kt_1}(j_k))}{(mN)^{\frac{k}{2}+1}}\prod_{l=1}^k
I_{[1,N]}(i+\sum\limits_{q=1}^l j_q)\delta_{0,\sum\limits_{q=1}^kj_q}.\\
\end{eqnarray*}
Similarly as in the proof of Theorem 3.2, we know that when $k$ is
odd, $m_{k,N}=o(1)$. For $m_{2k,N}$,
\begin{eqnarray*}
&
&\sum\limits_{i=1}^N\sum\limits_{j_1,...,j_{2k}=-b_N}^{b_N}\sum\limits_{t_1,...,t_{2k}=1}^m\frac{E(a_{t_1t_2}(j_1)\cdots
a_{t_{2k}t_1}(j_{2k}))}{(mN)^{k+1}}\prod_{l=1}^{2k}
I_{[1,N]}(i+\sum\limits_{q=1}^l
j_q)\delta_{0,\sum\limits_{q=1}^{2k}j_q}\\
&=&N^{-k-1}\sum\limits_{i=1}^N\sum\limits_{\pi\in\mathcal{P}_2(2k)}\sum\limits_{x_1,...,x_k=-b_N}
^{b_N}m^{k-1-f(\pi)}\prod_{l=1}^{2k} I_{[1,N]}(i+\sum_{q=1}^l
\epsilon_{\pi}(q)x_{\pi(q)})+o(1).
\end{eqnarray*}
So
\begin{eqnarray*}
& &m_{k,N}\\
&=&\frac{m^{k+1}}{mN(2mb_N)^{k}}\sum\limits_{i=1}^N\sum\limits_{\pi\in\mathcal{P}_2(2k)}\sum\limits_{x_1,...,x_k=-b_N}
^{b_N}m^{k-1-f(\pi)}\prod_{l=1}^{2k} I_{[1,N]}(i+\sum_{q=1}^l
\epsilon_{\pi}(q)x_{\pi(q)})+o(1)\\
&=&\frac{1}{N\cdot2^k\cdot
b_N^k}\sum\limits_{i=1}^N\sum\limits_{\pi\in\mathcal{P}_2(2k)}\sum\limits_{x_1,...,x_k=-b_N}
^{b_N}m^{k-1-f(\pi)}\prod_{l=1}^{2k} I_{[1,N]}(i+\sum_{q=1}^l
\epsilon_{\pi}(q)x_{\pi(q)})+o(1).
\end{eqnarray*}
Thus
\begin{eqnarray*}
\lim_{N\ra\iy}m_{2k,N}&=&\frac{1}{2^k}\sum\limits_{\pi\in\mathcal{P}_2(2k)}m^{k-1-f(\pi)}\int_{[0,1]\times
[-1,1]^{k}}\prod_{l=0}^{k}\mathrm{d}\, x_{l}\\
&=&\sum\limits_{\pi\in\mathcal{P}_2(2k)}m^{k-1-f(\pi)}.
\end{eqnarray*}
Let $m_{2k}=\lim\limits_{N\ra\iy}m_{2k,N}$ and
$m_{2k-1}=0(k\in\mathbb{N})$, then for any $k\in\mathbb{N}$ we have
$\lim\limits_{N\ra\iy}m_{k,N}=m_k$. It is easy to see that
$$m_{2k}\le\sum\limits_{\pi\in\mathcal{P}_2(2k)}m^{k-1-f(\pi)}\le(2k-1)!!\cdot
m^{k-1}$$ and then using Carleman's theorem (see [6]) we know that
$\mu_{_{X_N}}$ converges weakly to a distribution $\gamma^{(m)}$
which is determined by its even moments $\{m_{2k}\}$. Similarly as
in the the proof of Theorem 3.2 we know that
\begin{eqnarray*}
\frac{1}{N^4}E\Big(\big(\textrm{tr}X_N^k-E(\textrm{tr}X_N^k)\big)^4\Big)\le
B\cdot b_N^{-2}
\end{eqnarray*}
where $B$ is a constant and is independent of $N$. For $b_N\ge
C\cdot N^{\frac{1}{2}+\epsilon_0}$, we have
\begin{eqnarray*}
\frac{1}{N^4}E\Big(\big(\textrm{tr}X_N^k-E(\textrm{tr}X_N^k)\big)^4\Big)\le
B\cdot C^{-2}\cdot N^{-1-2\epsilon_0}.
\end{eqnarray*}
Therefore

\begin{eqnarray*}
\sum\limits_{N=1}^{\iy}\frac{1}{N^4}E\Big(\big(\textrm{tr}X_N^k-E(\textrm{tr}X_N^k)\big)^4\Big)<\iy
\end{eqnarray*}
and then $\mu_{_{X_N}}$ converges almost surely to the limit
distribution $\gamma^{(m)}$.
\end{proof}

\begin{remark}
In [15], the authors proved that the limit of eigenvalue
distribution for band Toeplitz random matrix with bandwidth
$b_N=o(N)$ is the standard normal distribution N(0,1). But in this
paper, the matrix is block Toeplitz and the conclusion will be
different. The expectation of $\gamma^{(m)}$ is $0$ and the variance
of $\gamma^{(m)}$ is
$\sum\limits_{\pi\in\mathcal{P}_2(2)}m^{-f(\pi)}=1$. The forth
moment of $\gamma^{(m)}$ is
$\sum\limits_{\pi\in\mathcal{P}_2(4)}m^{1-f(\pi)}=2+\frac{1}{m^2}$
which is not the forth moment of N(0,1) if $m\ne1$. So
$\gamma^{(m)}$ is not N(0,1) when $m\ne1$.
\end{remark}

\section{Convergence to Semicircle Law}
Suppose $\pi=\{\{a_1,b_1\},...,\{a_k,b_k\}\}\in\mathcal{P}_2(2k)$,
the system of linear equations determined by $\pi$ is (see
Definition 3.1)
\[
\begin{cases} t_{a_i}=t_{b_i+1}&1\le i\le k\\
t_{b_i}=t_{a_i+1}&1\le i\le k
\end{cases}
\]
where $t_{2k+1}=t_1$. This system can be rewritten as
\[
\begin{cases} t_{s_1(1)}=t_{s_1(2)}=\cdots =t_{s_1(r_1)}=t_{s_1(1)}\\
\cdots \\
t_{s_p(1)}=t_{s_p(2)}=\cdots =t_{s_p(r_p)}=t_{s_p(1)}
\end{cases}
\]
such that
$\{s_1(1),...,s_1(r_1)\}\cup\cdots\cup\{s_p(1),...,s_p(r_p)\}=\{1,2,...,2k\}$
and $\{s_i(1),...,s_i(r_i)\}\cap\{s_j(1),...,s_j(r_j)\}=\emptyset$
if $i\ne j$. We call $t_{s_i(1)}=t_{s_i(2)}=\cdots
=t_{s_i(r_i)}=t_{s_i(1)}$ a circle of $\pi$. We use $g(\pi)$ to
denote the number of the circles of $\pi$. For example, if
$\pi=\{\{1,2\},\{3,4\},\{5,6\},\{7,8\}\}\in\mathcal{P}_2(8)$, then
the system of linear equations determined by $\pi$ is
\[
\begin{cases}
t_1=t_3\\
t_2=t_2\\
t_3=t_5\\
t_4=t_4\\
t_5=t_7\\
t_6=t_6\\
t_7=t_1\\
t_8=t_8
\end{cases}.
\]
This system can be rewritten as
\[
\begin{cases}
t_1=t_3=t_5=t_7=t_1\\
t_2=t_2\\
t_4=t_4\\
t_6=t_6\\
t_8=t_8
\end{cases}.
\]
So this $\pi$ has five circles: $t_1=t_3=t_5=t_7=t_1$, $t_2=t_2$,
$t_4=t_4$, $t_6=t_6$ and $t_8=t_8$, thus $g(\pi)=5$. These circles
can be denoted by\\\\

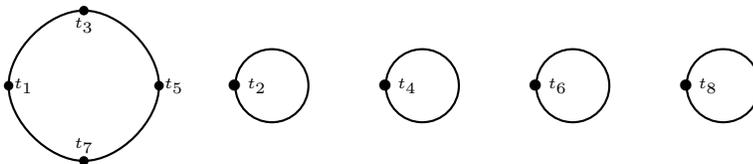
\psccurve[showpoints=true](0,0)(1,1)(2,0)(1,-1)
\rput(0.2,0){$\scriptscriptstyle{t_1}$}\rput(2.2,0){$\scriptscriptstyle{t_5}$}
\rput(1,0.8){$\scriptscriptstyle{t_3}$}\rput(1,-0.8){$\scriptscriptstyle{t_7}$}
\pscircle(3.5,0){0.5}\rput(3,0){$\bullet$}\rput(3.3,0){$\scriptscriptstyle{t_2}$}
\pscircle(5.5,0){0.5}\rput(5,0){$\bullet$}\rput(5.3,0){$\scriptscriptstyle{t_4}$}
\pscircle(7.5,0){0.5}\rput(7,0){$\bullet$}\rput(7.3,0){$\scriptscriptstyle{t_6}$}
\pscircle(9.5,0){0.5}\rput(9,0){$\bullet$}\rput(9.3,0){$\scriptscriptstyle{t_8}$}\\\\

For $f(\pi)$ denotes the number of independent equations of the
system, it is easy to see that $g(\pi)=2k-f(\pi)$ for any
$\pi\in\mathcal{P}_2(2k)$.

\begin{lemma}
For any $\pi\in\mathcal{P}_2(2k)$, $g(\pi)\le k+1$.
\end{lemma}

The proof of Lemma 5.1 can be found in [22].

We now review the concept of noncrossing partition (see [21]). A
partition $\pi\in\mathcal{P}(n)$ is called noncrossing if whenever
four elements $1\le a<b<c<d\le n$ are such that $a\sim_\pi c$ and
$b\sim_\pi d$, then $a\sim_\pi b\sim_\pi c\sim_\pi d$.

\begin{lemma}
For any $\pi\in\mathcal{P}_2(2k)$, $g(\pi)=k+1$ if and only if $\pi$
is noncrossing.
\end{lemma}

The proof of Lemma 5.2 can be found in [20].

\begin{theorem}
Suppose $\gamma_{_T}^{(m)}$ is defined as in Theorem 3.2. As
$m\ra\iy$, $\gamma_{_T}^{(m)}$ converges weakly to the semicircle
law $w(x)$, i.e.,
\[
w(x)=
\begin{cases} \frac{1}{2\pi}\sqrt{4-x^2}&|x|\le2,\\
0&|x|>2.
\end{cases}
\]
\end{theorem}

\begin{proof}[Proof of Theorem 5.3]
From Theorem 3.2 we know that the odd moments of $\gamma_{_T}^{(m)}$
are all zero and its even moments are
\[
m_{2k}(\gamma_{_T}^{(m)})=\sum\limits_{\pi\in\mathcal{P}_2(2k)}m^{k-1-f(\pi)}\int_{[0,1]\times
[-1,1]^{k}}\prod_{j=1}^{2k}I_{[0,1]}(x_{0}+\sum_{q=1}^{j}\epsilon_{\pi}(q)\,x_{\pi(q)})
\prod_{l=0}^{k}\mathrm{d}\, x_{l}
\]
where
\[
\epsilon_{\pi}(q)=
\begin{cases} 1&\text{if $q$ is the smaller element of $V_{\pi(q)}$},\\
-1&\text{if $q$ is the larger element of $V_{\pi(q)}$}.
\end{cases}
\]

From Lemma 5.1 and Lemma 5.2 we know that if
$\pi\in\mathcal{P}_2(2k)$, then $f(\pi)=k-1$ if $\pi$ is
noncrossing, otherwise $f(\pi)>k-1$. So we have
\[
\lim_{m\ra\iy}m_{2k}(\gamma_{_T}^{(m)})=\sum\limits_{\pi\in\mathcal{P}_2(2k)\atop
\pi \text{ is noncrossing}}\int_{[0,1]\times
[-1,1]^{k}}\prod_{j=1}^{2k}I_{[0,1]}(x_{0}+\sum_{q=1}^{j}\epsilon_{\pi}(q)\,x_{\pi(q)})
\prod_{l=0}^{k}\mathrm{d}\, x_{l}.
\]

From [4] we know that when $\pi$ is noncrossing,
\[
\int_{[0,1]\times
[-1,1]^{k}}\prod_{j=1}^{2k}I_{[0,1]}(x_{0}+\sum_{q=1}^{j}\epsilon_{\pi}(q)\,x_{\pi(q)})
\prod_{l=0}^{k}\mathrm{d}\, x_{l}=1.
\]

So

\[
\lim_{m\ra\iy}m_{2k}(\gamma_{_T}^{(m)})=\sum\limits_{\pi\in\mathcal{P}_2(2k)\atop
\pi \text{ is noncrossing}}1.
\]

From [11] we know that
$$\sharp\{\pi\in\mathcal{P}_2(2k)\big|\pi\text{ is
noncrossing}\}=C_k$$ where $C_k$ is Catalan number. So
\[
\lim_{m\ra\iy}m_{2k}(\gamma_{_T}^{(m)})=C_k.
\]
For $C_k$ is the $2k$-th moment of $w(x)$ whose odd moments are all
zero, we know that $\gamma_{_T}^{(m)}$ converges weakly to $w(x)$
thus the theorem is proved.
\end{proof}

\begin{remark}
If the assumptions in Theorem 4.3 are all satisfied, then it is easy
to see that as $m\ra\iy$, $\gamma^{(m)}$ converges weakly to the
semicircle law $w(x)$.
\end{remark}

\begin{remark}
For a block Hankel matrix as we discussed in Remark 3.3 we also have
that as $m\ra\iy$, $\gamma_{_H}^{(m)}$ converges weakly to the
semicircle law $w(x)$ because (see (3.9))
\begin{eqnarray*}
r(m,\pi)&=&\sharp\big\{(t_1,..,t_{2k})\,\big|\,1\le
t_{a_i},t_{b_i}\le
m;\,t_{a_i}=t_{(b_i+1)};\,t_{(a_i+1)}=t_{b_i}\,(1\le i\le
k)\big\}\\
&+&O(m^k)\\
&=&m^{2k-f(\pi)}+O(m^k)
\end{eqnarray*}
and $\{\pi\in\mathcal{P}_2(2k)\big|\pi\text{ is
noncrossing}\}\subset\mathcal{P}_2^1(2k)$ and the fact that the
integral in (3.8) is 1 when $\pi$ is noncrossing (see [4]). If the
blocks of $H_N$ are Hermitian matrices, the results would be
similar. If the blocks in a block Toeplitz matrix
$T=(A_{i-j})_{i,j=1}^N$ are complex matrices and
$A_{-s}=(\overline{A_s})^T$, then the results would also be similar.
\end{remark}

\begin{remark}
If the blocks in a block Toeplitz matrix $T=(A_{i-j})_{i,j=1}^N$ are
symmetric matrices and $A_{-s}=A_s$, just like those discussed in
[18], then the eigenvalue distribution of $T/\sqrt{mN}$ converges
almost surely to a distribution $\widetilde{\gamma_{_T}}^{(m)}$
which is determined by its even moments
\[
m_{2k}(\widetilde{\gamma_{_T}}^{(m)})=\sum\limits_{\pi\in\mathcal{P}_2(2k)}\frac{r(m,\pi)}{m^{k+1}}\int_{[0,1]\times
[-1,1]^{k}}\prod_{j=1}^{2k}I_{[0,1]}(x_{0}+\sum_{q=1}^{j}\epsilon_\pi(q)\,x_{\pi(q)})
\prod_{l=0}^{k}\mathrm{d}\, x_{l}
\]
where $r(m,\pi)$ is the same as (3.9). So we also have that as
$m\ra\iy$, $\widetilde{\gamma_{_T}}^{(m)}$ converges weakly to the
semicircle law $w(x)$.
\end{remark}

\section*{Acknowledgements}
The authors thank the anonymous referee for his heuristic question:
what is the iterated limit
$\lim\limits_{m\ra\iy}\lim\limits_{N\ra\iy}\mu_{_{X_N}}$. This
question leads to the semicircle law we discussed in Section 5.

\end{document}